\documentclass[11pt]{amsart}
\usepackage{epsfig}
\usepackage{graphicx, subfigure}
\usepackage{amsmath, amssymb,mathabx,mathrsfs}
\usepackage{color}
\usepackage[toc,page]{appendix}
\usepackage{comment}

\newtheorem{Theorem}{Theorem}
\newtheorem{Lemma}[Theorem]{Lemma}

\newtheorem{Proposition}[Theorem]{Proposition}

\newcommand{\eps}{\varepsilon}

\title{Stability interchanges in a curved Sitnikov problem}

\begin{document}

\author[L.Franco-P\'{e}rez]{Luis Franco-P\'{e}rez }
\email{lfranco@correo.cua.uam.mx}
\author[M. Gidea]{Marian  Gidea}
\email{marian.gidea@gmail.com}
\author[M. Levi]{Mark Levi}
\email{levi@math.psu.edu}
\author[E. P\'{e}rez-Chavela]{Ernesto  P\'{e}rez-Chavela}
\email{epc@xanum.uam.mx}

\address[UAM]{Departamento de Matem\'{a}ticas Aplicadas y Sistemas, UAM-Cuajimalpa, Av. Vasco de Quiroga 4871,  M\'{e}xico, D.F. 05348, M\'{e}xico.}
\address[YU]{Department of Mathematical Sciences, Yeshiva University, 245 Lexington Ave, New York, NY 10016, USA.}
\address[PSU]{Mathematics Department, Penn State University, University Park, PA 16802, USA.}
\address[ITAM]{Departamento de Matem\'{a}ticas, ITAM M\'exico, R\'io Hondo 1, Col. Progreso Tizap\'an, M\'exico D.F. 01080 .}

\begin{abstract}

We consider a curved Sitnikov problem, in which an infinitesimal particle moves on a circle under the gravitational influence of two equal masses in Keplerian motion within a plane perpendicular to that circle.  There are two equilibrium points, whose stability we are studying.  We show that one of the equilibrium points undergoes stability interchanges as the semi-major axis of the Keplerian ellipses approaches the diameter of that circle. To derive this result, we first formulate and prove a general theorem on stability interchanges, and then we apply it to our model.  The motivation for our model resides with the $n$-body problem in  spaces of constant curvature.
\end{abstract}


\maketitle

\section{Introduction}
\label{sec:introduction}

\subsection{A curved Sitnikov problem}\label{intro_confined} We consider the following curved Sitnikov problem:  Two bodies  of equal masses  (primaries)  move, under mutual gravity,  on Keplerian ellipses  about their center of mass. A third, massless particle is confined to a circle passing through the center of mass of the primaries, denoted by $P_0$, and   perpendicular to the  plane of motion of the primaries; the second intersection point of the circle with that plane  is denoted by $P_1$. We assume that the massless particle moves under the gravitational influence  of the primaries without affecting them. The dynamics of the massless particle has two equilibrium points, at $P_0$ and  $P_1$.  We  focus on the local dynamics near these two points, more precisely, on the dependence of the linear stability of these points on the parameters of the problem.

When the Keplerian ellipses are not too large or too small, $P_0$ is a local center and $P_1$ is a hyperbolic fixed point. When we increase the size of the Keplerian ellipses, as the distance  between $P_1$ and the closest ellipse  approaches zero, then $P_1$ undergoes stability interchanges.
That is, there exists a sequence of open, mutually disjoint intervals of values of  the semi-major axis of the Keplerian ellipses, such that, on each of these intervals the linearized stability of $P_1$ is strongly stable,  and each complementary interval contains values where the linearized stability  is not strongly stable, i.e., it is either hyperbolic or parabolic.
The length  of these intervals approaches zero when the semi-major axis of the Keplerian ellipses approaches the diameter of the circle on which the massless particle moves. This phenomenon is  the main focus in the paper.

{It is stated in \cite{Alfaro} and suggested by numerical evidence \cite{Hagel_Lothka,Kalas} that the
 linearized stability of the  point $P_0$ also undergoes stability interchanges when the size of the binary is kept fixed and the eccentricity  of the Keplerian ellipses approaches $1$.
 }

Stability interchanges of the type described above are ubiquitous in systems of varying parameters; they appear, for example,  in the classical Hill's equation and in the Mathieu equation \cite{Magnus}.  To prove the occurrence of this phenomenon in our curved Sitnikov problem, we first formulate a general result on stability interchanges for a    general class of simple mechanical systems. More precisely, we consider the motion of two bodies  --- one massive and one massless --- which are confined to a pair of curves and move under Newtonian gravity. We let the  distance between the two curves be controlled by some parameter $\lambda$.    We assume that the position of the infinitesimal particle that achieves the minimum distance between the curves is an equilibrium point. We show that, in the case when the  minimum distance between the two curves approaches zero, corresponding to $\lambda\rightarrow 0$, there existence a sequence of mutually disjoint open intervals  $(\lambda_{2n-1},\lambda_{2n})$,  whose lengths approach  zero as $\lambda\to 0$, such that whenever $\lambda \in (\lambda_{2n-1},\lambda_{2n})$ the linearized stability of the equilibrium point is strongly stable,  and each complementary interval contains values of $\lambda$ where the linearized stability  is not strongly stable.
From this result we derive the above mentioned stability interchange  result   for the curved Sitnikov problem.

The curved Sitnikov problem  considered in this paper is an extension of the classical Sitnikov problem described in Section~\ref{classical_sitnikov} (also, see e.g.,  \cite{Mos}). When the radius of the circle approaches infinity, in the limit we obtain the classical Sitnikov problem --- the infinitesimal mass moves along the line perpendicular to the plane of the primaries and passing through the center of mass. The equilibrium point $P_1$ becomes the point at infinity  and is of a degenerate  hyperbolic type. Thus,  stability interchanges of $P_1$ represent a new phenomenon that we encounter in the curved Sitnikov problem but not in the classical one.
{Also in the last case, it is well known that for $\eps=0$ the classical Sitnikov problem is integrable. In the case of the curved one, numerical evidence suggests that it is not  (see Figure \ref{fig:poincare}).}

The motivation for considering the curved Sitnikov problem  resides in  the $n$-body problem in spaces with constant curvature, and with models of planetary motions in binary star systems, as discussed in Section~\ref{n_body_curvature}.

\subsection{Classical Sitnikov problem}\label{classical_sitnikov}
We recall here the classical Sitnikov problem. Two bodies (primaries) of equal masses $m_1=m_2=1$ move in a plane on Keplerian ellipses  of eccentricity $\eps$ about their center of mass, and a third, massless particle moves on a line perpendicular to the plane of the primaries and   passing through their center of mass. By choosing the plane of the primaries the $xy$-plane and the line on which the massless particle moves the $z$-axis, the equations of motion of the massless particle can be written, in appropriate units, as
\begin{equation}\label{eqn:sitnikov1} \ddot z=-\frac{2z}{(z^2+r^2(t))^{3/2}},
\end{equation} where $r(t)$ is the distance from the primaries to their center of mass given by
\begin{equation}\label{eqn:sitnikov2} r(t)=1-\eps\cos u(t),
\end{equation}
where $u(t)$ is the eccentric anomaly in the Kepler problem. By normalizing the time we can assume that the period of the primaries is $2\pi$, and
\begin{equation}\label{eqn:sitnikov3} r(t)=(1-\eps\cos t)+O(\eps^2),
\end{equation}
for small $\eps$.

When   $\eps=0$, i.e., the primaries move on  a circular orbit and the dynamics of the massless particle is described by a 1-degree of freedom Hamiltonian and so is integrable. Depending on the energy level, one has the following types of solutions:  an equilibrium solution, when the particle rests at the center of mass of the primaries; periodic solutions around the center of mass; escape orbits, either parabolic, that  reach infinity with zero velocity, or hyperbolic, that  reach infinity with positive velocity.

When $\eps\in(0,1)$, the differential equation \eqref{eqn:sitnikov1} is non-autonomous and the system is non-integrable. Consider the case $\eps \ll 1$.
The system also has bounded and unbounded orbits, as well as  unbounded  oscillatory orbits and  capture orbits (oscillatory orbits are those for which   $\limsup_{t\to \pm\infty}|z(t)|=+\infty$ and $\liminf_{t\to \pm\infty}|z(t)|<+\infty$, and capture orbits  are those for which $\limsup_{t\to -\infty}|z(t)|=+\infty$ and
$\limsup_{t\to +\infty}|z(t)|<+\infty$).  In his famous paper about the final evolutions in the three body problem, Chazy introduced the term   {\it  oscillatory motions} \cite{Chazy}, although he did not find examples of these, leaving the question of their existence open. Sitnikov's model yielded the first example of oscillatory motions \cite{Sitnikov}.
There are many relevant works on this problem, including \cite {Alekseev1968a,Alekseev1968b,Alekseev1969,McGehee,Mos,Robinson,Dankowicz,GarciaPerezChavela,GorodetskiKaloshin}.

The curved Sitnikov problem introduced in Section~\ref{intro_confined} is a modification of the classical problem when the  massless particle moves on a circle rather than a line.  Here  we regard  the circle  as  a very simple restricted model of a space with constant curvature. In Subsection~\ref{n_body_curvature}
we introduce and summarize some aspects of  this problem.


\subsection{The $n$-body problem in spaces with
{constant} curvature}\label{n_body_curvature}
The   $n$-body problem on spaces with
{constant} curvature is a natural extension of the  $n$-body problem  in the Euclidean space; in either case  the gravitational law considered is Newtonian.  The extension was first proposed independently by the founders of hyperbolic geometry,  Nikolai Lobachevsky and J\'anos Bolyai. It was subsequently studied in the late 19th, early 20th century,  by Serret, Killing, Lipschitz, Liebmann, Schering, etc. Schr\"{o}dinger developed a quantum mechanical analogue of the Kepler problem on the two-sphere in 1940. The interest  in the problem was revived  by     Kozlov, Harin,  Borisov,  Mamaev,  Kilin, Shchepetilov, Vozmischeva, and others, in the 1990's. A  more recent surge of interest  was stimulated by the works on relative equilibria in spaces with constant curvature (both positive and negative)  by Diacu, P\'erez-Chavela, Santoprete,  and others, starting in the 2010's.  See \cite{Diacu} for a history of the  problem and a comprehensive list of references.

A distinctive aspect   of the $n$-body problem on curved spaces is that the lack of (Galilean) translational invariance results in the lack of  center-of-mass and linear-momentum integrals.
Hence, the study of the motion cannot be reduced to a  barycentric coordinate system.

As a consequence, the two-body problem on a sphere  can no longer be reduced to the
corresponding problem of motion in a central potential field, as is the case for   the Kepler problem in the Euclidean space. As it turns out, the two-body problem on the sphere is not  integrable \cite{Ziglin}.

Studying the three-body problem on spaces with curvature is also challenging.
Perhaps  the simplest  model is the restricted  three-body problem  on a circle. This was studied in  \cite{Fr}. First, they consider the motion of the two primaries on the circle, which  is integrable, collisions can be regularized, and all orbits can be classified into three different classes (elliptic, hyperbolic, parabolic). Then  they consider the motion of the massless particle  under the gravity of the primaries,   when one or both primaries are at a fixed position. They obtain once again a complete classification of all orbits of the massless particle.

In this paper we take the ideas from above one step further, by considering the curved Sitnikov problem,  with the  massless particle moving on a circle under the gravitational influence of two primaries that move on Keplerian ellipses in a plane perpendicular to that circle.   In the limit case, when the primaries are identified with one point, that is when the primaries coalesce into a single body, the Keplerian ellipses degenerate to a point, and  the limit problem coincides with the two-body problem on a circle described above.

While the motivation of this work is theoretical, there are possible connections with the dynamics of planets in binary star systems. About 20 planets outside of the Solar System have been confirmed to orbit about binary stars systems; since more than half of the main sequence stars have at least one stellar companion, it is expected that a substantial fraction of planets form in binary stars systems. The orbital dynamics of such planets can vary widely, with some planets orbiting one star and some others orbiting both stars. Some chaotic-like planetary  orbits have also been observed, e.g.  planet Kepler-413b orbiting Kepler-413 A and Kepler-413 B in the constellation Cygnus, which displays erratic precession. This planet's orbit is tilted away from the plane of binaries and deviating  from Kepler's laws. It is hypothesized that this tilt may be due to the gravitational influence of   a third star nearby \cite{Kostov2014}. Of a related interest is the relativistic version of the Sitnikov problem \cite{Kovacs}.
Thus, mathematical models like the one considered in this paper could  be helpful to understand  possible types of planetary orbits in  binary stars systems.

To complete this introduction, the paper is organized as follows: In Section~\ref{curved} we go deeper in the description of the curved Sitnikov problem, studying the limit cases and its general properties. In  Section~\ref{general} we present a general result on stability interchanges. In Section~\ref{section_stability} we show that the equilibrium points in the curved Sitnikov problem present stability interchanges. Finally, in order to have a self contained paper, we add an Appendix with general results (without proofs) from Floquet theory.

\section{The curved Sitnikov problem}\label{curved}
\subsection{Description of the model}
We consider two bodies with equal masses (primaries) moving under mutual Newtonian gravity on identical elliptical orbits of eccentricity $\varepsilon$, about their center of mass. For small values of $\varepsilon$, the distance $r(t)$ from either primary to the center of mass of the binary is given by
\begin{equation}\begin{split}\label{primarias}
     r_{\varepsilon}(t;r)=& r\rho(t;\varepsilon), \qquad r>0,\\\rho(t;\varepsilon)=&\left(1-\varepsilon\cos (u(t))\right)= \left(1-\varepsilon\cos (t)\right)+\mathcal{O}(\varepsilon^{2}),\end{split}
\end{equation}
{where $u(t)$ is the eccentric anomaly, which satisfies Kepler's equation $u-\eps \sin u=(2\pi/\tau)t$, where $\tau/2\pi$ is the mean motion\footnote{The mean motion is the time-average angular velocity over an orbit.} of the primaries. The expansion in \eqref{primarias} is convergent for $\eps < \eps_c = 0.6627...$; see \cite{Plummer}.}

A massless particle is confined  on a circle of radius $R$ passing through the center of mass of the binary and perpendicular to the plane of its motion. We assume that the only force acting on the infinitesimal particle is  the  component along the circle of the resultant of the gravitational forces exerted by the primaries. The motion of the primaries take place in the $xy$-coordinate plane and the circle with radius $R$ is in the $yz$-coordinate plane. See Figure \ref{problema}.

   \begin{figure}
                    \begin{center}
                  \includegraphics[width=8cm]{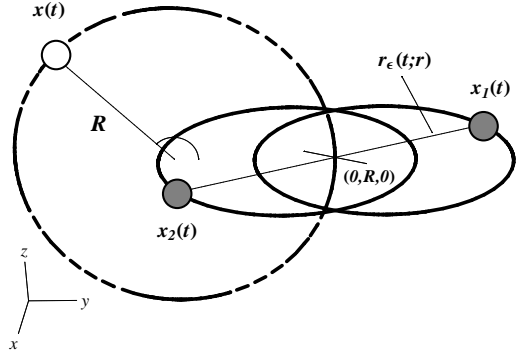}
                    \caption{\textit{The curved Sitnikov problem.}}
                    \label{problema}
                      \end{center}
   \end{figure}

  We place the center of mass at the point $(0,R,0)$ in the $xyz$-coordinate system. The position of the primaries are determined by the functions
  \begin{eqnarray*}
     \mathbf{x}_{1}(t)&=&(r_{\varepsilon}(t;r)\sin t ,R+r_{\varepsilon}(t;r)\cos t ,0),\\
     \mathbf{x}_{2}(t)&=&(-r_{\varepsilon}(t;r)\sin t ,R-r_{\varepsilon}(t;r)\cos t ,0).
  \end{eqnarray*}
{Note that $t=0$ corresponds to the passage of the primaries through the pericenter  at $y=R\pm r(1-\varepsilon)$ and  $t=\pi$ to the passage of the primaries through the apocenter at $y=R\pm r(1+\varepsilon)$; both peri- and apo-centers lie on the plane of the circle of radius $R$ in the $y$-axis.}

  The position of the infinitesimal particle is $\mathbf{x}(t)=(0,y(t),z(t))$ (taking into account the restriction of motion for the infinitesimal particle to the circle $y^{2}+z^{2}=R^{2}$). We will derive the equations of motion by computing the gravitational forces exerted by the primaries:
  \begin{eqnarray}\label{fuerza}
     \mathbf{F}_{\varepsilon}(y,t;R,r)&=&-\frac{\mathbf{x}-\mathbf{x}_{1}}{||\mathbf{x}-\mathbf{x}_{1}||^{3}}-
     \frac{\mathbf{x}-\mathbf{x}_{2}}{||\mathbf{x}-\mathbf{x}_{2}||^{3}}\\\nonumber
                  &=&-\frac{(-r_{\varepsilon}(t;r)\sin t ,y-R-r_{\varepsilon}(t;r)\cos t ,z)}
                  {||\mathbf{x}-\mathbf{x}_{1}||^{3}}\\\nonumber
                   &&-\frac{(r_{\varepsilon}(t;r)\sin t ,y-R+r_{\varepsilon}(t;r)\cos t ,z)}
                   {||\mathbf{x}-\mathbf{x}_{2}||^{3}}\,,
  \end{eqnarray}
  where the distance from the particle to each primary is
 \begin{eqnarray*}
     ||\mathbf{x}-\mathbf{x}_{1}||=
     \left[r^{2}_{\varepsilon}(t;r)+2R^{2}+2Rr_\eps(t;r)\cos t -2y(R+r_{\varepsilon}(t;r)\cos t )\right]^{1/2},\\
     ||\mathbf{x}-\mathbf{x}_{2}||=\left[r^{2}_{\varepsilon}(t;r) +2R^{2}-2Rr_\eps(t;r)\cos t -2y(R-r_{\varepsilon}(t;r)\cos t )\right]^{1/2}.
  \end{eqnarray*}

{We note that when  $r(1+\varepsilon)=2R$ the elliptical orbit of the primary with the  apo-center at $y<R$  crosses the circle of radius $R$, hence collisions between the primary and the infinitesimal mass are possible. Therefore we will restrict to $r<\frac{2R}{1+\varepsilon}$;  when $\varepsilon=0$, this means $r<2R$.
}

We write (\ref{fuerza}) in polar coordinates, that is $y=R\cos q$, $z=R\sin q$, and we obtain
  \begin{eqnarray}\label{fuerzaangulo}
     \mathbf{F}_{\varepsilon}(q,t;R,r)&=&-\frac{(-r_{\varepsilon}(t;r)\sin t ,R\cos q-R-r_{\varepsilon}(t;r)\cos t ,R\sin q)}{||\mathbf{x}-\mathbf{x}_{1}||^{3}}\\\nonumber
                   &&-\frac{(r_{\varepsilon}(t;r)\sin t ,R\cos q-R+r_{\varepsilon}(t;r)\cos t  ,R\sin q)}{||\mathbf{x}-\mathbf{x}_{2}||^{3}}.
                   \end{eqnarray}

  The origin  $q=0$ corresponds to the point $(0,R,0)$ in the $xyz$-coordinate system. Thus, the primaries move on elliptical orbits around this point.

  Next we will retain the component along the circle of the resulting force (\ref{fuerzaangulo}). That is, we will ignore the constraint force  that confines the motion of the particle to the circle, as this force  acts perpendicularly to the tangential component of the gravitational attraction force. The unit tangent vector to the circle of radius $R$ at $(0,R\cos( q),R\sin( q))$ pointing in the positive direction is given by $\mathbf{u}( q)=(0,-\sin( q),\cos( q))$. The component of the force $\mathbf{F}_{\varepsilon}( q,t;R,r)$ along the circle is computed as
  \begin{equation}\label{fuerzaangulorestringida}
     \mathbf{F}_{\varepsilon}( q,t;R,r)\cdot\mathbf{u}( q)=-\frac{(R+r_{\varepsilon}(t;r)\cos(t))\sin( q)}{||\mathbf{x}-\mathbf{x}_{1}||^{3}}-\frac{(R-r_{\varepsilon}(t;r)\cos(t))\sin( q)}{||\mathbf{x}-\mathbf{x}_{2}||^{3}}.
  \end{equation}

  The motion of the particle, as a Hamiltonian system of one-and-a-half degrees of freedom, corresponds to
  \begin{eqnarray}\label{sistemaestudio}
     \dot{ q}&=&p\,,\\\nonumber
     \dot{p}&=&f_{\varepsilon}(q,t;R,r),
  \end{eqnarray}
  where
  \begin{eqnarray*}
     f_{\varepsilon}( q,t;R,r)&:=&\mathbf{F}_{\varepsilon}( q,t;R,r)\cdot\mathbf{u}( q)\,,\\
     ||\mathbf{x}-\mathbf{x}_{1}||&=&\left[ r^{2}_{\varepsilon}(t;r) +2R(1-\cos  q)(R+r_{\varepsilon}(t;r)\cos t) \right]^{1/2}\,,\\
     ||\mathbf{x}-\mathbf{x}_{2}||&=&\left[ r^{2}_{\varepsilon}(t;r)
     +2R(1-\cos  q)(R-r_{\varepsilon}(t;r)\cos t)\right]^{1/2}\,,
  \end{eqnarray*}

Hence
\begin{equation}
H_\eps(q,p,t;R,r)=\frac{p^2}{2}+V_\eps(q,t;R,r),
\end{equation}
where the potential is given by  \begin{align} V_\eps(q,t;R,r)&= -\frac{1}{R}\left(\frac{1}{||\mathbf{x}-\mathbf{x}_{1}||}+\frac{1}{||\mathbf{x}-\mathbf{x}_{2}||}\right)
.\end{align}

\subsection{Limit cases}

The curved Sitnikov problem can be viewed as  a link between the classical Sitnikov problem and the Kepler problem on the circle, mentioned in  Section~\ref{sec:introduction}.

\subsubsection{The limit $R\rightarrow\infty$.}

   We express \eqref{fuerzaangulorestringida}  in terms of the arc length $w=R q$, obtaining

  \begin{multline}\label{fuerzalongituddearco}
      {f}_{\varepsilon}(w,t;R,r)=-\frac{(R+r_{\varepsilon}(t;r)\cos t )\sin\left(w/R\right)}
     {\left[r^{2}_{\varepsilon}(t;r) +2R(1-\cos \left(w/R\right))(R+r_{\varepsilon}(t;r)\cos t) \right]^{\frac{3}{2}}}\\
      -\frac{(R-r_{\varepsilon}(t;r)\cos t )\sin\left(w/R\right)}
      {\left[r^{2}_{\varepsilon}(t;r) +2R(1-\cos \left(w/R\right))(R-r_{\varepsilon}(t;r)\cos t) \right]^{\frac{3}{2}}}
  \end{multline}
  which we can write   in a suitable form as
  \begin{multline*}
      {f}_{\varepsilon}(w,t;R,r)=-\frac{w\frac{\sin\left(w/R\right)}{\left(w/R\right)}
     \left(1+\frac{r_{\varepsilon}(t;r)}{R}\cos t \right)}{
     \left[r^{2}_{\varepsilon}(t;r) +2w^2\frac{(1-\cos \left(w/R\right))}{\left ( w/R\right)^2}\left(1+\frac{r_{\varepsilon}(t;r)}{R}\cos t\right)  \right]^{\frac{3}{2}}}\\
      -\frac{w\frac{\sin\left(w/R\right)}{\left(w/R\right)}
     \left(1-\frac{r_{\varepsilon}(t;r)}{R}\cos t \right)}{
     \left[r^{2}_{\varepsilon}(t;r) +2w^2\frac{(1-\cos \left(w/R\right))}{\left ( w/R\right)^2}\left(1-\frac{r_{\varepsilon}(t;r)}{R}\cos t\right)  \right]^{\frac{3}{2}}}.
  \end{multline*}
  Letting $R$ tend to infinity we obtain
   \[
      \lim_{R\rightarrow\infty} {f}_{\varepsilon}(w,t;R,r)=-\frac{2w}{\left(r^{2}_{\varepsilon}(t;r)+w^{2}\right)^{3/2}}\,,
   \]
   which is the classical Sitnikov Problem.

\subsubsection{The limit $r\rightarrow0$.}\label{section_limit_0}

When we take the limit $r\rightarrow0$ in (\ref{fuerzaangulorestringida}) we are fusing the primaries into a large mass at the center of mass and we obtain a  two-body problem on the circle.
The component force along the circle corresponds to
\begin{equation}\label{eq:A1}
\lim_{r\rightarrow0}f_{\varepsilon}( q,t;R,r)=-\frac{\sin( q)}{\sqrt{2}R^{2}(1-\cos( q))^{3/2}}\,.
\end{equation}
This problem was studied in  \cite{Fr}   with a different  force given by
\begin{equation}\label{eq:A2}
 -\frac{1}{R q^{2}}+\frac{1}{R(2\pi- q)^{2}},
 \end{equation}
 the distance between the large mass and the particle is measured by the arc length (in that paper the authors assume that  $R=1$).
The potential of the force (\ref{eq:A1}) is
\begin{equation*}
         V_{1}(q_{1})=-\frac{1}{R^2\sqrt{2}(1-\cos (q_{1}))^{1/2}},
\end{equation*}
where $q_1$ denotes the angular coordinate, and the potential for (\ref{eq:A2}) is
      \begin{equation*}
         V_{2}(q_{2})=-\frac{1}{Rq_{2}}-\frac{1}{R(2\pi-q_{2})},
      \end{equation*}
where $q_2$ denotes the angular coordinate.
Each problem  defines an autonomous system with Hamiltonian
\begin{equation}\label{energia}
         H_{i}(p_{i},q_{i})=\frac{1}{2}p_{i}^{2}+V_{i}(q_{i})\,,
\end{equation}
taking $p_{1}=dq_{1}/dt$, $p_{2}=dq_{2}/dt$ and $i=1,2$.
Let $\phi^i_{t}$  be the   flow of the Hamiltonian  $H_{i}$, and let  $ A_i $ denote the phase space, $ i=1, \  2 $

{Using that all orbits are determined by the energy relations given by (\ref{energia}), it is not difficult to define a homeomorphism $g:A_1 \to A_2$ which maps orbits of system (\ref{eq:A1}) into orbits of system (\ref{eq:A2}). In the same way we can define a homeomorphism $h:A_2 \to A_1$ which in fact is
$g^{-1}$. This shows the $C^0$ equivalence of the respective flows.}
One can show that the two corresponding flows are $C^0$--equivalent.


We recall from \cite{Fr} that the solutions of the two-body problem on the circle (apart from the  equilibrium  antipodal  to the fixed body)  are  classified in three families (elliptic, parabolic and hyperbolic solutions) according to their energy level. The elliptic solutions come out of a collision, stop instantaneously, and reverse their path back to the collision with the fixed body. The parabolic  solutions come out of a collision and approach the equilibrium as $ t \rightarrow \infty $. Hyperbolic motions  comes out of a collision with the fixed body, traverse  the whole circle and return to a collision.

We remark that the two limit cases $R\to\infty$ and $r\to 0$ are not equivalent. Indeed,  in the case $r\to 0$ the resulting system is autonomous,
the point $q=0$ is a singularity for the system, and the  point $q=\pi$ is a hyperbolic fixed point, while in the case $R\to\infty$  the resulting system is non-autonomous (for $\eps\neq 0$),    the point $q=0$ is a fixed point of elliptic type, and the point $q=\infty$ is a degenerate hyperbolic periodic orbit.

\subsection{General properties}\label{generalproperties}

\subsubsection{Extended phase space, symmetries, and equilibrium points}
It is clear that, besides the limit cases $R\to\infty$ and $r\to 0$, the dynamics of the system depends only on the ratio $r/R$, so we can fix $R=1$ and study the  dependence of the global dynamics on $r$ where  $0<r<2$.
In this case using (\ref{sistemaestudio})  and (\ref{primarias}) we get
\begin{multline}\label{fuerzaparametrogamma}
      {f}_{\varepsilon}(q,t;r)=-\frac{(1+r\rho(t;\varepsilon)\cos(t))\sin( q)}
     {\left[r^{2}\rho^2(t;\varepsilon) +2(1-\cos  q)(1+r\rho(t;\varepsilon)\cos t)
     \right]^{3/2}}\\
     -\frac{(1-r\rho(t;\varepsilon)\cos(t))\sin( q)}{\left[r^{2}\rho^2(t;\varepsilon) +2(1-\cos  q)(1-r\rho(t;\varepsilon)\cos t)
     \right]^{3/2}}.
  \end{multline}

To study the non-autonomous system \eqref{sistemaestudio} we will make the system autonomous by
introducing the time as an extra dependent variable
\begin{equation}\label{eqn:autonomous}
     \mathcal{X}_{\varepsilon}(q,p,s;r)=\left\{
        \begin{array}{rcl}
           \dot{q}&=&p\\
           \dot{p}&=& {f}_{\varepsilon}(q,s;r)\\
           \dot{s}&=&1
        \end{array}\right..
   \end{equation}

This vector field is defined on $[0,2\pi]\times\mathbb{R}\times[0,2\pi]$, where we identify the boundary points of the closed intervals. The flow of $\mathcal{X}_ \varepsilon$ possesses symmetries defined by the functions
   \begin{eqnarray*}
          \mathbb{S}_{1}(q,p,s)&=&(-q,-p,s)\,,\\
          \mathbb{S}_{2}(q,p,s)&=&(q,p,s+2\pi)\,,\\
          \mathbb{S}_{3}(q,p,s)&=&(q+2\pi,p,s)\,,\\
          \mathbb{S}_{4}(q,p,s)&=&(q,-p,-s)\,
      \end{eqnarray*}
  in the sense that
         \begin{itemize}
         \item $\mathbb{S}_{1}(\mathcal{X}_{\varepsilon}(q,p,s))=\mathcal{X}_{\varepsilon}(\mathbb{S}_{1}(q,p,s))$,
         \item $\mathcal{X}_{\varepsilon}(q,p,s)=\mathcal{X}_{\varepsilon}(\mathbb{S}_{2}(q,p,s))$,
         \item $\mathcal{X}_{\varepsilon}(q,p,s)=\mathcal{X}_{\varepsilon}(\mathbb{S}_{3}(q,p,s))$
         \item $\mathbb{S}_{4}(\mathcal{X}_{\varepsilon}(q,p,s))=-\mathcal{X}_{\varepsilon}(\mathbb{S}_{4}(q,p,s))$,
      \end{itemize}
  as can be verified by a direct computation.

{The function $\mathbb{S}_{1}$ describes  the  symmetry respect to the trajectory $(0,0,s)$, $\mathbb{S}_{2}$ and $\mathbb{S}_{3}$  describe the bi-periodicity of  ${f}_{\varepsilon}(q,s;r)$ and $\mathbb{S}_{4}$ describes the reversibility of the system.}

System \eqref{sistemaestudio} has two equilibria  $(0,0)$ and $(\pi,0)$, which correspond to periodic orbits for $\mathcal{X}_{\varepsilon}$.

While the classical Sitnikov equation     is autonomous for $ \varepsilon = 0$, our equation \eqref{eqn:autonomous} is not, and thus we expect it to be non--integrable, as is borne out by numerical simulation.  Figure \ref{fig:poincare}  shows a Poincar\'e section corresponding to $s=0$ (mod $2\pi$), with $\varepsilon=0$ and $r=1$. This simulation suggests that the invariant KAM circles
coexist with chaotic regions.
   \begin{figure}
                    \begin{center}
                  \includegraphics[width=9.5cm]{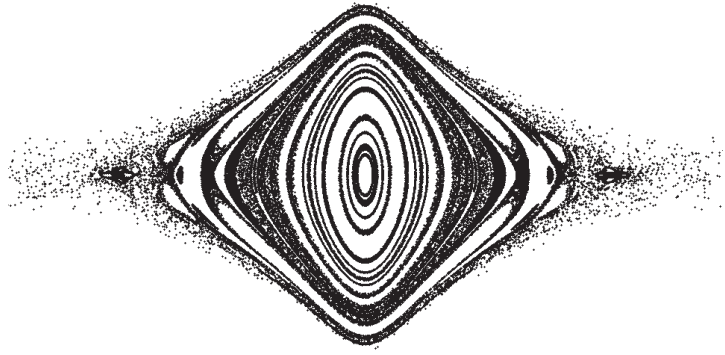}
                    \caption{\textit{Poincar\'e section for the curved Sitnikov problem, for $\varepsilon=0$ and $r=1$.}}
                    \label{fig:poincare}
                      \end{center}
   \end{figure}

{In the sequel, we will  analyze the dynamics around the  equilibrium points $(\pi,0)$ and $(0,0)$.  One important phenomenon that we will observe is that both equilibrium points  undergo stability interchanges as parameters are varied. More precisely, when 
{$\varepsilon$  sufficiently  small is kept fixed} and $r\rightarrow \frac{2R}{1+\varepsilon}$, the point $(\pi,0)$ undergoes infinitely many changes in stability, and when $r$ is kept fixed  and $\varepsilon\to 1$, the point $(0,0)$ undergoes infinitely many changes in stability.}

In the next section we will first prove a general result.

\section{A general result on stability interchanges}\label{general}
{In this section we switch to a  more general mechanical model which exhibits stability interchanges. We consider  the motion under mutual gravity of an infinitesimal  particle  and a heavy mass each constrained to its own curve and moving under gravitational attraction, and  study the linear stability of the equilibrium point corresponding to the closest position between the particles along the curves they are moving on. In Section~\ref{section_stability}, we will apply this general result to  the equilibrium point $P_1$ of the curved Sitnikov problem described in Section \ref{intro_confined}.  The fact that in the curved Sitnikov problem there are two heavy masses, rather than a single one as considered in this section, does not change the validity of the stability interchanges result, since, as we shall see,  what it ultimately matters is the time-periodic gravitational potential acting on the infinitesimal  particle.}  


To describe the setting of this section, consider   a particle   constrained to a curve $ {\bf x} = {\bf x} (s, \lambda ) $ in $ {\mathbb R}  ^3 $, where $s$ is the arc length along the curve and $\lambda$ is a parameter with values in some interval $ [0, \lambda_0] $, $\lambda_0> 0 $.
Another (much larger) gravitational mass undergoes a {\it  prescribed } periodic  motion according to  $ {\bf y} = {\bf y} (t, \lambda) =  {\bf y} (t+1, \lambda )$; see Figure  \ref{fig:generalcase_A}. We assume that the mass of the particle at $ {\bf x} (s) $ is negligible compared to the mass at  ${\bf y}(t) $, treating the particle at ${\bf x} (s) $ as  massless.

 \begin{figure}[t!]
	\begin{center}
	  \includegraphics[width=9.5cm]{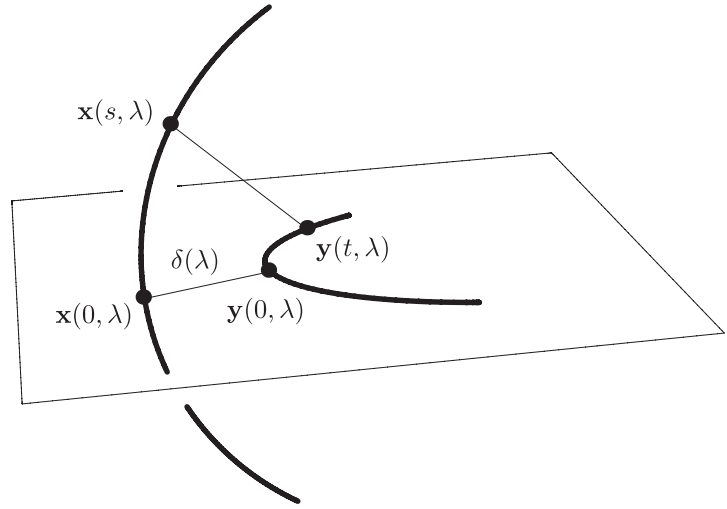}
	\caption{ An infinitesimal particle with coordinates ${\bf x}(s)$ is constrained to a curve and moves under the influence of the larger mass with coordinates ${\bf y}(t)$. }
	\label{fig:generalcase_A}
\end{center}
\end{figure}

To write the equation of motion for the unknown coordinate $s$ of the massless particle,  let:
\begin{equation}
	{\bf z} (s,t, \lambda ) = {\bf x} (s, \lambda) - {\bf y} (t, \lambda)
	\label{eq:q}.	
\end{equation}
Assume that $ s=0, \  t=0 $ minimize the distance between the two curves:
\begin{equation}
	 | {\bf z} (0,0, \lambda) | = \min_{s,t} |{\bf z} (s,t, \lambda)|\buildrel{def}\over{=} \delta( \lambda ),
	 \label{eq:min}
\end{equation}
for all $ \lambda \in [0, \lambda_0]$,  and that this minimum point is non-degenerate with respect to $t$, in the sense that \begin{equation}\frac{\partial^2 }{\partial t^2}|z(0,t,\lambda)|_{\mid t=0} \neq 0 .\label{eq:min02}
\end{equation}

Moreover, we make the following orthogonality assumption:
\begin{equation}\label{eqn:orthogonality}
\dot{\bf x}(0,\lambda)\cdot \dot{\bf y}(0,\lambda)=0,\, \textrm{ for all }\lambda\in[0,\lambda_0].
\end{equation}

{In the sequel we will study the case  when the minimum distance $\min_{s,t} |{\bf z} (s,t, \lambda)| \rightarrow 0 $, that is, the shortest distance from the orbit ${\bf y}(t,\lambda)$ of the massive body  to the curve ${\bf x}(s,\lambda)$
drawn by the infinitesimal mass approaches $0$ as $\lambda\to\lambda_0$. See Figure~\ref{fig:generalcase_A}.
In the curved Sitnikov problem, this corresponds to the case when $r\to \frac{2R}{1+\varepsilon}$ ($r\to 2R$ when 
{ $\varepsilon  = 0$).}
}

To write the equation of motion for $s$, we note that
 the Newtonian gravitational potential of the particle at $ {\bf x} (s) $ is a function of $s$ and $t$  given by
\begin{equation}
	U( s, t, \lambda  )  =- | {\bf z} (s,t, \lambda )  | ^{-1},
	\label{eq:U} 	
\end{equation}
and the evolution of $s$ is    governed by the Euler--Lagrange equation
$ \frac{d}{dt} L_{\dot s}-L_s=0 $ with the Lagrangian
\[
	L = \frac{1}{2} \dot s ^2 - U( s, t, \lambda  ),
\]
leading to\footnote{to explain this form of the Lagrangian, we note that  the equations for our massless particle are   obtained by taking the limit of the particle of small mass $m$; for such a particle, in the ambient potential $U$,  the Lagrangian is
$ \frac{1}{2} m \dot s ^2    -m U( s, t, \lambda  ) $ -- the factor $m$ in front of $U$ is due to the fact that $U$ is the potential energy of the unit mass. Dividing the Euler--Lagrange equation  by $m$ gives   (\ref{eq:governing.eq}).}
\begin{equation}
	 \ddot s + U ^\prime ( s , t,\lambda)  = 0, \  \textrm{ where }  ^\prime = \frac{\partial}{\partial s} .
	 \label{eq:governing.eq}
\end{equation}

Note that $s=0 $ is an equilibrium for any $\lambda$, since $ U ^\prime (0,t, \lambda ) = 0 $ for all $ t $ and for all $\lambda$, according to   (\ref{eq:min}).
Linearizing   (\ref{eq:governing.eq})
around the equilibrium $ s=0 $ we obtain
\begin{equation}
	 \ddot S+ a(t, \lambda  ) S = 0, \  \  a(t+1, \lambda ) = a (t, \lambda ),    	
	 \label{eq:lineq}
\end{equation}
where $a(t, \lambda  ) =  U ^{\prime\prime} (0,t,\lambda)$.

\medskip

We have the following general result:

{\begin{Theorem} \label{thm:thm1}
Assume that  \eqref{eq:min}, \eqref{eq:min02},
 \eqref{eqn:orthogonality}   hold, that $ {\bf x} $ and $ {\bf y} $ are both bounded in the $ C^2 $--norm uniformly in $\lambda$, and
\begin{equation}\label{eqn:A}\min_{s,t}\|x(s,\lambda)-y(t,\lambda)\|\to 0 \textrm{ as }\lambda\to 0.
\end{equation}
Then there exists an infinite sequence
\begin{equation}
	\lambda_1 > \lambda_2 \geq \lambda _3> \lambda  _4\geq \cdots >
	\lambda_ {2n-1}>\lambda_ {2n} \geq  \cdots	 \rightarrow 0
	\label{eq:sequence}
\end{equation}
such that  the equilibrium solution $ s=0 $ of   (\ref{eq:governing.eq})   is linearly  strongly stable\footnote{The equilibrium solution is linearly  strongly stable if the Floquet multipliers of   (\ref{eq:lineq})  lie on the unit circle and are not real, or equivalently, if the linearized system lies in the interior of the set of stable systems.} for all $ \lambda  \in (\lambda_{2n}, \lambda  _{2n-1}) $. Furthermore, each complementary $\lambda $--interval  contains points where the linearized equilibrium is not strongly stable, i.e. is either hyperbolic or  parabolic.
\end{Theorem}
}

The proof of  Theorem \ref{thm:thm1} relies on Lemmas \ref{lem:arginf}, \ref{lem:arg} and \ref{lem:uniformlimit} stated below.

 \begin{Lemma} \label{lem:arginf} Consider the linear system
 \begin{equation}
	 \ddot x + a(t, \lambda ) x=0,
	\label{eq:hill1}
\end{equation}
where $a(t,\lambda)$ is a continuous function  of $ t\in [0,1] $, $ \lambda \in (0, \bar \lambda] $, where $ \bar \lambda > 0 $.  Let $ z(t, \lambda) = x+ i \dot x $, where $x=x(t, \lambda ) $ is    a nontrivial solution of   (\ref{eq:hill1}). Assume that there exists an
  interval $ [t_0( \lambda ), t_1( \lambda ) ]\subset [0,1]$, possibly depending on $\lambda$, such that
\begin{equation}
	\lim_{ \lambda \downarrow 0} \arg z(t, \lambda)\biggl|_{t_0}^{t_1} \rightarrow -\infty.
	\label{eq:arginf}
\end{equation}
Here $\arg z(t, \lambda )$ is defined as a continuous function of $t$. Although this choice of $\arg$ is unique only modulo $ 2 \pi $,
its increment as stated in equation (\ref{eq:arginf}) is uniquely defined.

Then there exists a sequence  $\{ \lambda_k\}_{k=0}^ \infty$ satisfying (\ref{eq:sequence})  such that the Floquet matrix of   (\ref{eq:lineq})  is strongly stable
for all  $\lambda  \in (\lambda_{2n}, \lambda  _{2n-1}) $, for any $ n > 0 $. Furthermore, every complementary $\lambda$--interval contains
values of $\lambda$ for which the Floquet matrix is not strongly stable.
\end{Lemma}

Proof of this lemma can be found in \cite{Levi}.

 \begin{Lemma} \label{lem:arg}
Consider the linear system
\begin{equation}
	 \ddot x + a(t) x=0,
	\label{eq:hill}
\end{equation}    and assume that $ a(t) > 0 $ on some interval
$ [t_0, t_1] $.  For any (nontrivial) solution $ x(t) $ the corresponding  phase vector
$ z(t ) = x + i \dot x $ rotates by
\begin{equation}
	\theta [z]\buildrel{def}\over{=} \arg z(t) \biggl|_{t_0}^{t_1}\leq -\min_{[t_0,t_1] } \sqrt{ a(t) } (t_1-t_0) + \pi.
	\label{eq:argestimate}
\end{equation}
\end{Lemma}

\begin{proof}
Writing the differential equation $ \ddot x + a (t) x = 0 $ as a system $ \dot x = y, \  \dot y  = - a(t)x $, we obtain  (using complex notation): 
\[
	\dot \theta =  \frac{d}{dt} \arg (x+iy)= - \frac{d}{dt}  \hbox{Im} ( \ln z) = - \hbox{Im} \frac{\dot z  }{z} =
	-\hbox{Im} \frac{(y-iax)(x-iy)}{x ^2 + y ^2 } =  -(a \cos^2  \theta+ \sin ^2 \theta ).
\]
We conclude that for any solution of   (\ref{eq:hill}), the angle   $\theta=\arg (x+iy)$ satisfies
\[
	\dot \theta \leq  -(a_m\cos ^2 \theta + \sin^2 \theta ), \  \  \hbox{for}  \  \   t\in [t_0,t_1] , \  \  \hbox{where}  \  \   a_m=\min_{[t_0,t_1] }a(t).
\]
To invoke comparison estimates, consider $ \bar\theta (t) $ which satisfies
\begin{equation}
		\frac{d}{dt} \bar   \theta =  -(a_m \cos^2 \bar \theta + \sin^2 \bar\theta ), \  \  \bar \theta (t_0) = \theta (t_0).
		\label{eq:thetabar}
\end{equation}
By the comparison estimate, we conclude:
\begin{equation}
	 \theta \biggl|_{t_0}^{t_1}\leq   \bar\theta \biggl|_{t_0}^{t_1},
	\label{eq:comparison}
\end{equation}
 and the proof of the lemma will be complete once we show that $ \bar \theta $ satisfies the estimate   (\ref{eq:argestimate}).
To that end we  consider a solution of
\[
	\ddot {\bar x } + a_m \bar x = 0
\]
with the initial condition satisfying
\begin{equation}
	\arg(\bar x(t_0)+i\dot   {\bar x}(t_0)) = \bar \theta (t_0).
	\label{eq:ic}
\end{equation}    This  solution is of the form
\[
	{\bar x}(t)= A\cos  (\sqrt{ a_m} t- \varphi ), \  \  A ={\rm const.} ,
\]
where $\varphi$ is chosen so as to satisfy   (\ref{eq:ic}).

Since $ \arg( \bar x+ i\dot   {\bar x}) $ satisfies
the same differential equation as $ \bar \theta $, and since the  initial conditions match, we conclude that
$ \bar \theta(t)  = \arg( \bar x+ i\dot   {\bar x})$, so that\begin{equation}
	\bar \theta(t)  =
	\arg( \cos  (\sqrt{ a_m }  t- \varphi) - i\sqrt{  a_m }\sin (\sqrt{  a_m } t-  \varphi ) )=
	- (\sqrt{a_m} t- \varphi )+ \widehat{ \pi /2},
	\label{eq:barest}
\end{equation}
where $ \widehat X $ denotes a quantity whose absolute value does not exceed $X$.  In other words, $ \bar \theta(t) $ is given by a linear function with coefficient $ -  \sqrt{ a_m} $, up to an error
$ < \pi /2$.
The last inequality is due to the fact that the complex numbers $\cos(\sqrt{a_m}t-\phi)-i\sin(\sqrt{a_m}t-\phi)$ and
$\cos(\sqrt{a_m}t-\phi)-i\sqrt{a_m}t\sin(\sqrt{a_m}t-\phi)$ lie in the same quadrant, so the difference between their arguments is no more than $\pi/2$.

Therefore, over the interval $ [t_0,t_1] $ the function $ \bar \theta $ changes by  the amount $ \sqrt{ a_m}(t_1-t_0) $ with the error of at most $ \frac{\pi}{2} + \frac{\pi}{2} = \pi$:
\begin{equation}
	\bar \theta  \biggl|_{t_0}^{t_1}\leq -\sqrt{  a_m } (t_1-t_0) + \pi;
	\label{eq:barest1}
\end{equation}
restating this more formally,   (\ref{eq:barest})  implies
\[
	\bar \theta(t_1)<  -  (\sqrt{ a_m} t_1- \varphi)+ \frac{\pi}{2} , \  \  \  \bar \theta(t_0)>  -  (\sqrt{ a_m}t_0- \varphi)- \frac{\pi}{2}.
\]
Subtracting the second inequality from the first gives   (\ref{eq:barest}).

Substituting  (\ref{eq:barest1}) into  (\ref{eq:comparison}) yields (\ref{eq:argestimate})  and completes the proof of  Lemma \ref{lem:arg}. \hfill $\diamondsuit$
\end{proof}

\begin{Lemma} \label{lem:uniformlimit}
Consider the potential $U$ defined by equation (\ref{eq:U}).    Assume that the functions
${\bf x}={\bf x}(s,\lambda)$, ${\bf y}={\bf y}(t,\lambda)$ satisfy
\begin{itemize}
\item[(i)]   the minimum $\min_{s,t} |{\bf z}(\cdot,\cdot,\lambda) |=\delta(\lambda)$ is non-degenerate and   is achieved at $s=t=0$, where ${\bf z}(s,t,\lambda)={\bf x}(s,\lambda)-{\bf y}(t,\lambda)$,

\item[(ii)]  $||{\bf z}(\cdot,\cdot,\lambda) ||_{C^2} \leq M $  uniformly in 
{$0<\lambda < \lambda_0$}, and

\item[(iii)] $\delta ( \lambda ) \rightarrow 0$  as $ \lambda \rightarrow 0 $.
\end{itemize}

Then there exists time $\tau= \tau (\lambda)$ (approaching zero as $\lambda \rightarrow 0$) such that

\begin{equation}
	\lim_{ \lambda \rightarrow 0} \left(\tau(\lambda) \cdot \min_{|t|\leq\tau (\lambda)}  \sqrt{U ^{\prime\prime} (0,t,\lambda)} \right)	\rightarrow \infty.
	\label{eq:lemma4}
\end{equation}
\end{Lemma}

\begin{proof}
Differentiating   (\ref{eq:U})  with respect to $s$ twice, we get
\begin{equation}
	U^{\prime \prime} (0,t, \lambda )   = \left[\frac{({\bf z}^\prime  \cdot {\bf z}^\prime + {\bf z} \cdot {\bf z}  ^{\prime\prime})
	({\bf z} \cdot {\bf z} )  -3({\bf z} \cdot {\bf z} ^\prime )^2 }
	{({\bf z} \cdot {\bf z})^{5/2}}\right]_{\mid s=0}.
	\label{eq:Upp}
\end{equation}
We now estimate all the dot products in the above expression to obtain a lower bound.

First,
\begin{equation}
	{\bf z} ^\prime \cdot {\bf z} ^\prime = 1,
	\label{eq:1}
\end{equation}
since  $s$ is the arc length, and from here ${\bf z} ^\prime \cdot {\bf z} ^{\prime \prime} = 0$. Now, to  estimate
$ {\bf z}\cdot {\bf z} $ and $ {\bf z}\cdot {\bf z}^\prime$ we observe that  ${\bf z}\cdot {\bf z}^\prime= \frac{1}{2} ({\bf z}\cdot {\bf z})^\prime  $  and we note that the first expression, as a  function of $t$ with $ s=0 $ fixed has a   minimum at $t=0$ that we call $\delta ^2 $, and that the second function vanishes at $ t=0 $.  Applying Taylor's formula with respect to $t$ we then have
 \begin{equation}
	{\bf z} \cdot {\bf z}  =\delta^2 + \widehat{k  t ^2}, \hspace{.7cm}  {\bf z} \cdot {\bf z}^\prime =  \frac{1}{2}
	({\bf z} \cdot {\bf z})^\prime  = \widehat {\ kt\ },
	\label{eq:zz}
\end{equation}
where the constant $k$ is determined by  the $ C^2$--norm $M$ of $ {\bf z} $. For the remaining dot product we have (still keeping $ s=0 $ and $t$ arbitrary):
\begin{equation}
	|{\bf z} \cdot {\bf z} ^{\prime\prime}| \leq | {\bf z} | |{\bf z} ^{\prime\prime}  |    \buildrel{  (\ref{eq:zz}) }\over{\leq}  M
	\sqrt{ \delta ^2 + k t ^2 }.
	\label{eq:zz''}
\end{equation}

Using the above estimates in  (\ref{eq:Upp}),  we obtain
\begin{equation}
	U ^{\prime \prime}(0,t, \lambda ) \geq
	\frac{(1- M \sqrt{  \delta^2 + k     t ^2 })\delta^2- 3k ^2 t^2 }
	{(\delta^2 + k  t ^2)^{5/2}} .
	\label{eq:inequality1}
\end{equation}

Now we restrict  $t$ to have  $\delta ^2 + kt ^2\leq 2 \delta^2 $; this  guarantees that the denominator in  (\ref{eq:inequality1}) does not exceed $ (2 \delta )^{5/2}$; to bound the numerator, we further restrict $t$ so that the dominant part
$ \delta ^2 - 3 k ^2 t^2 \geq \frac{1}{2} \delta ^2 $, thus bounding
the numerator from below by
\[
	(  \delta ^2 - 3 k ^2 t ^2 )  -  M  \sqrt{2 \delta ^2} \delta^2 \geq \frac{1}{2} \delta ^2 - M  \sqrt{2} \delta ^3 > \frac{1}{4} \delta ^2
\]
if $ \delta   $ is sufficiently small.
Summarizing, we restricted $t$ to
\begin{equation}
	| t | \leq c \delta   \buildrel{def}\over{=} \tau ( \lambda ), \  \  \hbox{where}  \  \ c = \min(k^{-1/2} ,
	(k  \sqrt{ 6} ) ^{-1}  ),
	\label{eq:tau}
\end{equation}
and showed that for all such $t$ and for $ \delta$ small enough
\begin{equation}
	U^{\prime \prime}(0,t, \lambda ) \geq
	\frac{ \frac{1}{4} \delta ^2  }
	{(2\delta^2 )^{5/2}}  = \frac{c_1}{\delta ^3},
	\label{eq:inequality2}
\end{equation}
where $ c_1=2 ^{-9/2} $.
With $ \tau $ defined in   (\ref{eq:tau}) we obtain
  $\lim_{\tau\to 0} \left( \tau\cdot\min_{|t| \leq \tau} \sqrt{a(t,\lambda)}\right) = \infty$, thus completing the proof of the lemma.
 $\diamondsuit$
\end{proof}

{\bf Proof of Theorem \ref{thm:thm1}}
Consider  the linearized equation
\begin{equation}
	\ddot x + U^{\prime \prime} (0,t, \lambda)x = 0,
	\label{eq:lineq1}
\end{equation}
{and consider the phase point   $ z=x+i \dot x $ of a nontrivial solution.
 Lemma \ref{lem:arg} gives us the rotation estimate (\ref{eq:argestimate}) for any time interval $ [t_0, t_1] $; let this interval be
 $ [- \tau ( \lambda ) ,\tau ( \lambda )] $ where $ \tau ( \lambda ) $  is taken from the statement of Lemma  \ref{lem:uniformlimit}. We then have  from   (\ref{eq:argestimate}):
  (\ref{eq:argestimate})
 	\begin{equation}
	\theta [z]\buildrel{def}\over{=} \arg z(t) \biggl|_{t_0}^{t_1}\leq
	- 2\tau(\lambda )\cdot\min_{|t|\leq \tau(\lambda )}\sqrt{U^{\prime \prime}(0,t, \lambda)} + \pi.
	\label{eq:argestimate1}
\end{equation}
According to the conclusion  (\ref{eq:lemma4})  of   Lemma \ref{lem:uniformlimit},
$ \theta \rightarrow - \infty $ as $ \lambda \rightarrow 0 $. This satisfies the condition   (\ref{eq:arginf})
Lemma \ref{lem:arginf}, which now applies and its conclusion comples the proof of  Theorem \ref{thm:thm1}.  \hfill $\diamondsuit$
}
\section{Stability of the equilibrium points in the curved Sitnikov problem}\label{section_stability}
The system \eqref{sistemaestudio} has two equilibria  $(0,0)$ and $(\pi,0)$, which correspond to periodic orbits for $\mathcal{X}_{\varepsilon}$.
The associated linear system around the fixed point $(q_{*},p_{*})$ can be
written as
   \begin{equation}\label{partelineal}
      \dot{\mathbf{v}}=A(t)\mathbf{v}\,,\qquad \mathbf{v}=\left(\begin{array}{c}x \\y \\\end{array} \right)\,,
   \end{equation}
   with
   \begin{equation*}A(t)=\left.\left(
               \begin{array}{cc}
                 0 & 1 \\
                 \frac{\partial  {f}_{\varepsilon}}{\partial q} & 0 \\
               \end{array}
             \right)\right|_{q=q_{*}}.
   \end{equation*}

Let $X(t)$ be a fundamental matrix solution of system (\ref{partelineal}) given by
   \[
    X(t)=\left(
           \begin{array}{cc}
             x_{1}(t) & x_{2}(t) \\
             y_{1}(t) & y_{2}(t) \\
           \end{array}
         \right),
    \]
     with the initial condition $X(0)=I$, the identity matrix;  $x_{1}$ is an even function and $x_{2}$ and odd one since $\partial  {f}_{\varepsilon}/\partial q$ is an even function with respect to $t$. The monodromy matrix  is given by $X(2\pi)$ and we denote $\lambda_{1},\lambda_{2}$ its eigenvalues, the Floquet multipliers associated to (\ref{partelineal}).
   These are given by
   \begin{equation}\label{valorespropios}
      \lambda_{1},\lambda_{2}=\frac{x_{1}(2\pi)+y_{2}(2\pi)\pm\sqrt{\left(x_{1}(2\pi)+y_{2}(2\pi)\right)^{2}-4}}{2}\,,
   \end{equation}
   the trace  $\textrm{Tr}(X(2\pi))=x_{1}(2\pi)+y_{2}(2\pi)$ determines the linearized dynamics around the fixed point. Moreover, since the function $(\partial {f}_{\varepsilon}/\partial q)_{\mid q=q^*}$  is an even function, we know $x_{1}(2\pi)=y_{2}(2\pi)$ (see \cite{Magnus}) and then
   \begin{equation}\label{valorespropiosenpi}
      \lambda_{1},\lambda_{2}=y_{2}(2\pi)\pm\sqrt{\left(y_{2}(2\pi)\right)^{2}-1}\,.
   \end{equation}
   To emphasize the dependence on the parameters $\eps, r$ we write
   \begin{equation}\label{funciondelta}
      y_{2}(2\pi;\varepsilon,r)=y_2(2\pi).
      \end{equation}
 Thus the linear stability of the equilibrium is
   \begin{enumerate}
      \item Elliptic type: $| y_{2}(2\pi;\varepsilon,r) |<1$.
      \item Parabolic type: $|y_{2}(2\pi;\varepsilon,r)|=1$.
      \item Hyperbolic type: $|y_{2}(2\pi;\varepsilon,r)|>1$.
   \end{enumerate}
   Since the Wronskian is equal to $1$ for all $t$, then
   \[W(\varepsilon, r)=(y_{2}(2\pi;\varepsilon,r))^{2}-x_{2}(2\pi;\varepsilon,r)y_{1}(2\pi;\varepsilon,r)=1,\]
   and we have:
   \begin{equation}\label{wronskiano}
      \left(y_{2}(2\pi;\varepsilon,r)\right)^{2}=1+x_{2}(2\pi;\varepsilon,r)y_{1}(2\pi;\varepsilon,r)\,.
   \end{equation}
   From this last expression it follows  that, in  the parabolic case, the periodic orbit corresponding to the equilibrium point $(\pi,0)$ is associated to   $x_{1}(t)$ or to $x_{2}(t)$ (or to both).


\subsection{Stability of the equilibrium point $(\pi,0)$.  Case $\eps=0$.} In this subsection we consider the system \eqref{sistemaestudio}-\eqref{fuerzaparametrogamma} for the case $\varepsilon=0$, namely when the primaries are following circular trajectories. We also fix $R=1$. The extended vector field under study is
   \begin{equation}\label{funcionestudio_0}
     \mathcal{X}_{0}(q,p,t;\gamma)=\left\{
        \begin{array}{rcl}
           \dot{q}&=&p\\
           \dot{p}&=&{f}_{0}(q,t;r)\\
           \dot{t}&=&1
        \end{array}\right.
         \end{equation}
  where
  \begin{equation}\label{eq:f0}
     {f}_{0}(q,t;r):= -\frac{(1+r\cos(t))\sin(q)}{||\mathbf{x}-\mathbf{x}_{1}||^{3}}-\frac{(1-r\cos(t))\sin(q)}
     { ||\mathbf{x}-\mathbf{x}_{2}||^{3}}\,,\\
 \end{equation}
     \begin{eqnarray*}
     ||\mathbf{x}-\mathbf{x}_{1}||&=&\left[r^{2}+2+2 r\cos(t)-2\cos (q)-2r\cos(q)\cos(t))\right]^{1/2}\,,\nonumber\\
     ||\mathbf{x}-\mathbf{x}_{2}||&=&\left[r^{2}+2-2 r\cos(t)-2\cos (q)+2r\cos(q)\cos(t))\right]^{1/2}\,,\nonumber
  \end{eqnarray*}
  with $0<r<2$. We observe that in this case, function (\ref{eq:f0}) is $\pi$--periodic (remember that ${f}_{\varepsilon}(q,t;r)$ is $2\pi$--periodic if $\varepsilon >0$).


   The linear system associated to (\ref{funcionestudio_0}) around the fixed point $(\pi,0)$ is defined by the function
   \begin{equation}\label{parcialenpi}
      \frac{\partial {f}_{0}}{\partial q}(\pi,t;r)=\frac{(1+r\cos(t))}{\left[r^{2}+4+4r\cos(t)\right]^{3/2}}
     +\frac{(1-r\cos(t))}{\left[r^{2}+4-4r\cos(t)\right]^{3/2}}\,,
 \end{equation}
   which is $C^{1}$ respect to $t$ and $r$.

\begin{Lemma}\label{monotone} The function $\frac{\partial {f}_{0}}{\partial q}(\pi,t;r)$  is monotone decreasing with respect to   $r$, that is,  for all $t\in[0,\pi)$
   \begin{equation*}
      \label{equality2}
      \frac{\partial {f}_{0}}{\partial q}(\pi,t;r_{1})>\frac{\partial {f}_{0}}{\partial q}(\pi,t;r_{2})
   \end{equation*}
   if $r_{1}<r_{2}$ (see Figure \ref{partelinealenpifig}).
\end{Lemma}

\begin{proof} Let be $F(t,r)=\frac{\partial {f}_{0}}{\partial q}(\pi,t;r)$, then
by straightforward computation we get
\begin{eqnarray*}
\frac{\partial F}{\partial r}(t,r) &=& \frac{(\cos(t))[r^2+4r\cos(t)+4] - 3(1+r\cos(t))[r+2\cos(t)]}{\left[r^{2}+4+4r\cos(t)\right]^{5/2}} \\
     &+& \frac{(- \cos(t))[r^2-4r\cos(t)+4] - 3(1-r\cos(t))[r-2\cos(t)]}{\left[r^{2}+4-4r\cos(t)\right]^{5/2}} \\
&\leq & \frac{-4r\cos(t)^2 - 6r}{\min \{ \left[r^{2}+4+4r\cos(t)\right]^{5/2},\left[r^{2}+4-4r\cos(t)\right]^{5/2} \}} < 0.
    \end{eqnarray*}
\hfill $\diamondsuit$
\end{proof}


   \begin{figure}
   \begin{center}
                  \includegraphics[width=7cm]{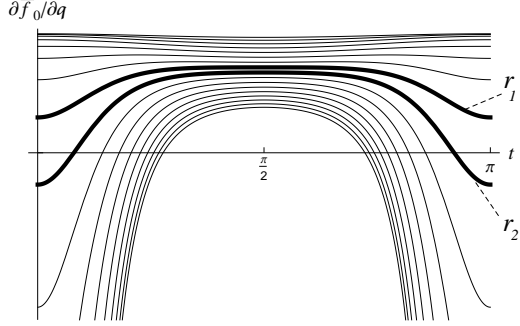}
                  \caption{\textit{Plots for $\partial \bar{f}_{0}/\partial q$ varying $r$ from $0$ to $2$.}}
                  \label{partelinealenpifig}
   \end{center}
   \end{figure}

\begin{Proposition}\label{proposicion2}
   The equilibrium point $(\pi,0)$ is of hyperbolic type if $r\leq\left(\sqrt{17}-3\right)^{1/2}=1.059\dots$.
\end{Proposition}
   \begin{proof}
We first show that $F(t,r)\geq 0$ for all $t$ if and only if $r\leq\left(\sqrt{17}-3\right)^{1/2}$.
We compute \begin{equation}\begin{split}\label{eqnderf}\frac{\partial F}{\partial t}(t,r)=r\sin t\left [-\frac{1}{(r^2+4r\cos t+4)^{1/2}} +\frac{1}{(r^2-4r\cos t+4)^{1/2}} \right.\\ \left.
+\frac{6(1+r\cos t)}{(r^2+4r\cos t+4)^{5/2}}-\frac{6(1-r\cos t)}{(r^2-4r\cos t+4)^{5/2}}\right ].\end{split}\end{equation}

We have $F(t,r)=F\left( \pi-t, r \right)$ so it is enough to restrict $t\in[0,\frac{\pi}{2}]$. Note that $\frac{\partial F}{\partial t}(t,r)=0$ for $t=0,\pi/2$.

For $0<r<2$, we have \[F \left (\frac{\pi}{2},r \right)=\frac{2}{(r^2+4)^{3/2}}>0,\]
and \[F(0,r)=\frac{1+r}{(2+r)^{3}}+\frac{1-r}{(2-r)^{3}}=\frac{-2(r^4+6r^2-8)}{(4-r^2)^{3}}.\]

It follows immediately that $F(0,r)<0$ if $r>(\sqrt{17}-3)^{1/2}$.
If $r\cos t\leq 1$ then \eqref{parcialenpi} implies $F(t,r)\geq 0$. Hence $F(t,r)\geq 0$ for all $t$ provided $r\leq 1$.
Let $1<r\leq (\sqrt{17}-3)^{1/2}$. If $t\in[0,\cos^{-1}(1/r)]$, which is equivalent to $r\cos t>1$, then \begin{equation*}\begin{split} \frac{1}{(r^2-4r\cos t+4)^{1/2}}
-\frac{1}  {(r^2+4r\cos t+4)^{1/2}}\geq 0,\\
\frac{6(1+r\cos t)}{(r^2+4r\cos t+4)^{5/2}}-\frac{6(1-r\cos t)}{(r^2-4r\cos t+4)^{5/2}}\geq 0.\end{split}\end{equation*}
Therefore \eqref{eqnderf} implies that  $F(t,r)$ is increasing in $t$ for $t\in [0,\cos^{-1}(1/r)]$, and, since $F(0,r)\geq 0$, it follows that $F(t,r)\geq 0$ for all $t$.

Now let  $x(t)=x_1(t)  + x_2(t)$ be a particular solution of system (\ref{partelineal}); by hypothesis it satisfies $x(0)=1, \dot{x}(0)=1$. From \eqref{partelineal} with $\varepsilon = 0$ we obtain
$$ \ddot{x} = \frac{\partial {f}_{0}}{\partial q}(\pi,t;r) x \quad {\rm or} \quad \ddot{x} - F(t,r)x = 0.$$

Then, since  $F(t,r)\leq 0$ for all $t$, we can apply directly Lyapunov's instability criterion (see for instance page 60 in \cite{Ces}) to show that $(\pi,0)$ is of hyperbolic type.  $\diamondsuit$\end{proof}

 We can in fact estimate the first value $r_1$ of $r$ at which the equilibrium point $(\pi,0)$ becomes of parabolic type for the first time. 
 {The idea is to extend slightly the result beyond $r=\left(\sqrt{17}-3\right)^{1/2}$,} and then find the maximum of the $r$ for which Proposition \ref{proposicion2} holds. In the way we have to do straightforward analytic but tedious computations that we decide to avoid in this paper. Finally we get that the value of $r$ to have the first parabolic solution is  $r_1 \approx 1.2472\cdots$.



   \begin{figure}
   \begin{center}
       \subfigure{\includegraphics[width=6.5cm]{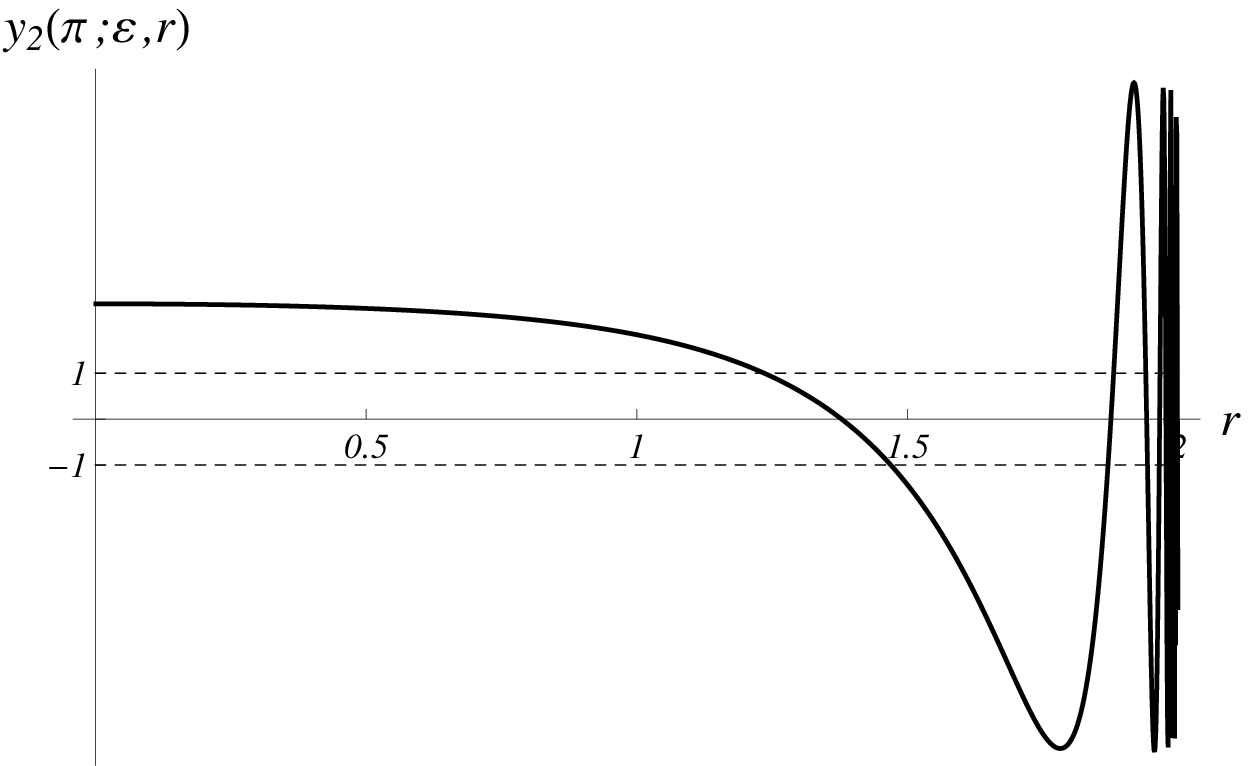}}
       \subfigure{\includegraphics[width=6.5cm]{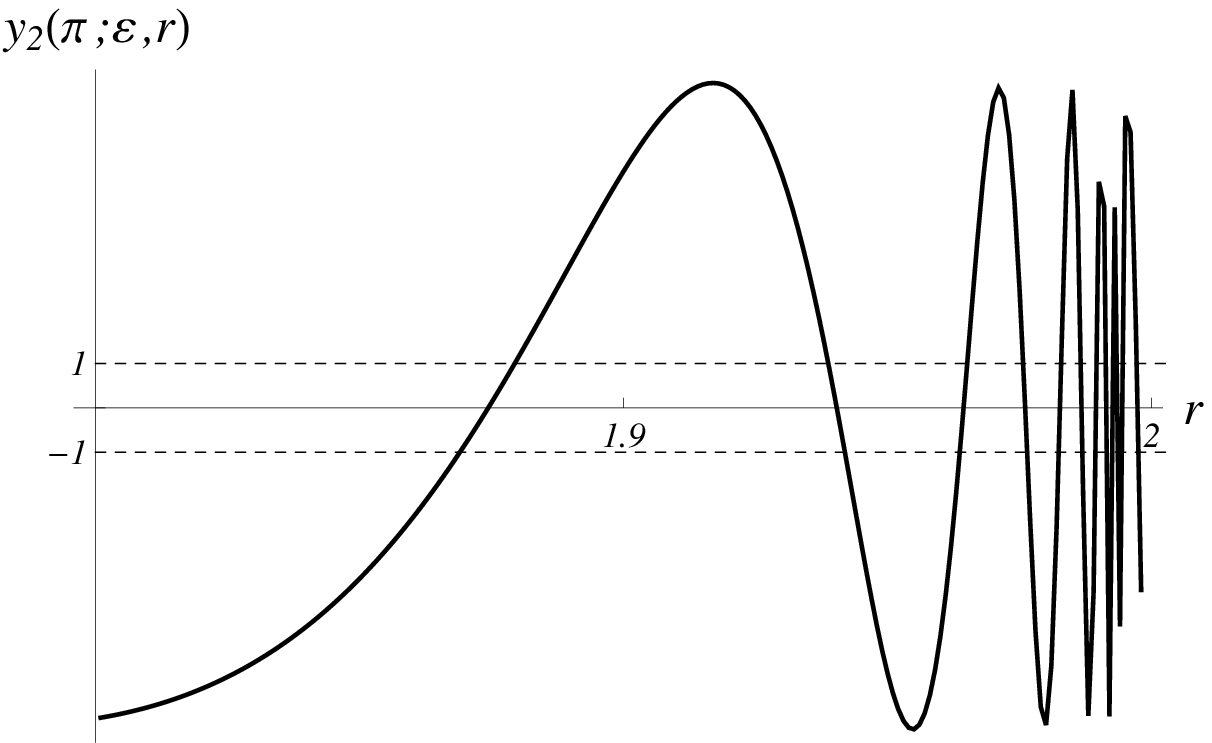}}
       \caption{\textit{Graphs of $y_{2}(\pi;\varepsilon,r)$. The plot on the right is a magnification of the one on the left for r close to 2}}
       \label{estabilidad}
   \end{center}
   \end{figure}

In Figure \ref{estabilidad} we show a couple of numerical simulations which illustrate the stability interchanges of the equilibrium point $(\pi,0)$ for $r \in (0,2)$ and $\varepsilon = 0$. The figure on the right hand side is a plot for $r$ close to 2.

\subsection{Stability of the equilibrium point $(\pi,0)$.  Case $\eps \neq 0$.}

In this case, for every $0<r<(\sqrt{17}-3)^{1/2}$, $\partial f_0/\partial q (\pi,t;r)>0$,  hence $\partial f_\eps/\partial q (\pi,t;r)>0$ for all $\eps>0$ sufficiently small (depending on $r$). Therefore  $(\pi,0)$  remains an equilibrium point of hyperbolic type in the case when the primaries move on Keplerian ellipses of sufficiently small eccentricity $\varepsilon>0$.

We remark   that  Theorem \ref{thm:thm1}  does not depend of the shape of the curves where ${\bf x}(x)$ or ${\bf y}(t)$ are moving, in other words we can apply Theorem \ref{thm:thm1} independently if the primaries are moving on a circle or on ellipses of eccentricity $\varepsilon>0$. Thus we obtain the following result on  stability interchanges:

\begin{Theorem}\label{thm:Sit1}
In the curved Sitnikov problem, let us fix any $\varepsilon\in [0,1)  $,   let $R=1$, and consider  $r$ (the semi-major axes of the Keplerian ellipses traced out by the primaries) as the parameter. As $r$ approaches $\frac{2}{1+\varepsilon}$, the distance between a Keplerian ellipse and the circle of the massless particle approaches zero. There exists a sequence   $r_n\uparrow \frac{2}{1+\varepsilon}$ satisfying
\begin{equation}\label{secesion}
	r_0 \leq r_1 < r _2 \leq  r _3 \cdots <
	r_{2n}\leq r_{2n+1} < \cdots,
\end{equation}
such that the equilibrium point $(\pi, 0)$ of equation (\ref{partelineal})
is strongly stable for $ r \in (r_{2n-1}, r _{2n}) $
 and not strongly stable for some $r$  in the complementary intervals $(r_{2n }, r _{2n+1}) $. In other words, the equilibrium point $(\pi, 0)$ loses and then regains its strong stability infinitely many times as $ r $ increases towards $\frac{2}{1+\varepsilon}$.
\end{Theorem}
{\begin{proof}[Proof of Theorem \ref{thm:Sit1}]
We verify that Theorem \ref{thm:thm1} applies. The curve ${\bf x}(s,\lambda)$ of Theorem \ref{thm:thm1} is represented by the circle of radius $R=1$,  the curve $ {\bf y} (t,\lambda)$ is represented by the orbit of the primary that gets closer to $(\pi,0)$ (when $\varepsilon =0$ the primaries are co-orbital), and the parameter $\lambda$ corresponds to $r$.  The planes of  the two curves are perpendicular, as in \eqref{eqn:orthogonality}, and the minimum distance between the curves, given by \eqref{eq:min}, is non-degenerate as in \eqref{eq:min02}, and it corresponds to the infinitesimal mass being at $y=-1$ and the closest primary to this point being at $y=1-r(1+\varepsilon)$. Thus, the minimum distance is $\delta(r)=2-r(1+\varepsilon)$,  and it approaches $0$ when $r\to \frac{2}{1+\varepsilon}$. To apply Theorem \ref{thm:thm1} we only need to verify that ${\bf x}$ and ${\bf y}$ are bounded in the $C^2$-norm uniformly in $r$. This is obviously true for ${\bf y}(t,r)$ since the motion of the primary is not affected by the motion of the infinitesimal mass; it is also true for ${\bf x}(s,r)$ since $s=$arclength and the motion lies on the circle of radius $R=1$, hence $\|{\bf x}(s,r)\|=1$ (the radius of the circle), $\|{\bf x}^\prime(s,r)\|=1$ (the unit speed of a curve parametrized by arc-length), and $\|{\bf x}^{\prime\prime}(s,r)\|=1$ (the curvature of the circle). Hence the conclusion of Theorem \ref{thm:thm1} follows immediately.
 $\diamondsuit$\end{proof}
}

\subsection{Stability of the equilibrium point $q=(0,0)$.}
{The  linear stability of  $(0,0)$ is the same as of the barycenter in the classical Sitnikov problem, since
\[\dfrac{\partial{f_\varepsilon}}{\partial q}(0)=-\frac{2R}{r^3\rho(t;\varepsilon)^3}.
\]}
The stability of this point can be treated very similarly to that of the origin for Hill's equation, so in this analysis we  use results from that theory.

As in the study of the other equilibrium point we start with the case  $\varepsilon = 0.$ Here the function $f_0(q,t;r)$
 defined in equation (\ref{sistemaestudio}) around $q=(0,0)$ is given by
  \begin{equation}\label{funcionencero}
     \frac{\partial f_{0}}{\partial q}(q)=-\frac{2}{r^{3}}q+\frac{9+r^{2}+9r^{2}(\cos(t))^{2}}{3r^{5}}q^{3}+\mathcal{O}(q^{5})\,.
  \end{equation}
  The local dynamics is determined by the linear part. The eigenvalues are on the unit circle, given by $\sigma_{1,2}=\pm i\sqrt{2/r^{3}}$, and the Floquet multipliers, which come from the monodromy matrix
  \[
     X(\pi)=
     \left(
       \begin{array}{cc}
         \cos \left(\sqrt{\frac{2}{r^{3}}}\pi\right)& \left(\sqrt{\frac{r^{3}}{2}}\right)\sin \left(\sqrt{\frac{2}{r^{3}}}\pi\right) \\
         -\left(\sqrt{\frac{2}{r^{3}}}\pi\right)\sin \left(\sqrt{\frac{2}{r^{3}}}\pi\right) & \cos \left(\sqrt{\frac{2}{^{3}}}\pi\right) \\
       \end{array}
     \right), \qquad X(0)=I\,,
  \]
  are $\lambda_{1,2}=e^{\pm i\sqrt{2/r^{3}}\pi}$.

Hence $q=(0,0)$ is of elliptic type if $\sqrt{2/r^{3}}\neq k$ for any $k\in\mathbb{Z}$, and it is of parabolic type if $\sqrt{2/r^{3}}=k$ for some $k\in\mathbb{Z}$. In fact, in the parabolic case, if $\sqrt{2/r^{3}}=2m$, $m\in\mathbb{Z}$, then  there exists a $\pi$-periodic solution, and if $\sqrt{2/r^{3}}=2m+1$, $m\in\mathbb{Z}$, then there exists a $2\pi$-periodic solution.

\begin{Proposition}\label{proposicion8}
   The equilibium point $q=(0,0)$ of  the system defined by   \eqref{funcionencero}  is stable  for all $r\in(0,2)$.
\end{Proposition}
   \begin{proof}
      The linear part possess Floquet multipliers which place the system in either the elliptic or the parabolic case.
      In the last case, when $\sqrt{2/r^{3}}=k$ for some $k\in\mathbb{Z}$, there exists two independent eigenvectors associated to the eigenvalues $\lambda_{1,2}$. Then the conjugacy  of
      $A$ and $\pm I$ implies that the monodromy matrix is $A = \pm I$ and $q=(0,0)$ is stable (see \cite{Ortega2} for more details).

When the origin is of elliptic type for the linear system,
we observe that the coefficient of the term of order $3$ in equation (\ref{funcionencero}) is greater than zero for all $r\in(0,2)$, then  we can apply directly Ortega's theorem (see Appendix and \cite{Ortega}), and therefore we obtain that the equilibrium point $(0,0)$ is stable for the whole system, that is, including the nonlinear part. $\diamondsuit$\hfill
\end{proof}

We note that, in the case when   $\varepsilon=0$, stability interchanges in a weak sense appear as $r\to 0$,  since $(0,0)$ switches between  elliptic type when $r\neq (2/k^2)^{1/3}$ and parabolic type when $r= (2/k^2)^{1/3}$, $k\in\mathbb{Z}^+$, as noted before.

{In the case when $\varepsilon\neq 0$, as we mentioned earlier, the linear stability of $(0,0)$ is the same as in the classical Sitnikov problem, so it only depends on the eccentricity parameter $\varepsilon$.
The papers \cite{Alfaro,Hagel_Lothka,Kalas} state that there are stability interchanges when the size $r$ of the binary is kept fixed and the eccentricity $\varepsilon$ of the Keplerian ellipses approaches $1$.
We should point out that Theorem \ref{thm:thm1} does not apply to this case since the function  $r_\varepsilon(t,r)$ describing the motion of the primaries --- corresponding to $y(t,\lambda)$ in Theorem \ref{thm:thm1} ---   does not remain bounded in
the $C^2$ norm uniformly in $\varepsilon$.}


\bigskip

\appendix

\section{Floquet Theory} In order to have a self contained paper we add this appendix with the main results on Floquet theory, must of them are very well known for people in the field.

 Consider the linear system
  \begin{equation}\label{linearsystem}
     \dot{\mathbf{x}}=A(t)\mathbf{x}\,,\qquad \mathbf{x}\in\mathbb{R}^{2}\,,
  \end{equation}
  where $A(t)$ is a $T$-periodic matrix-valued function. Let $X(t)$ be the fundamental matrix solution
  \begin{equation}\label{fundmatrix} X(t)=\left(
            \begin{array}{cc}
              x_{1}(t) & x_{2}(t) \\
              y_{1}(t) & y_{2}(t) \\
            \end{array}
          \right)
  \end{equation}
  with the initial condition $X(0)=I$, the identity matrix. Let $\lambda_{1}$ and $\lambda_{2}$ be the eigenvalues (Floquet multipliers) of the monodromy matrix matrix $X(T)$   and let $\mu_{1},\mu_{2}$ (Floquet exponents) be such that $\lambda_{1}=e^{\mu_{1}T}, \lambda_{2}=e^{\mu_{2}T}$.

  \begin{Theorem}[Floquet's theorem]
     Suppose $X(t)$ is a fundamental matrix solution for (\ref{linearsystem}), then
        \[ X(t+T)=X(t)X(T)\]
     for all $t\in\mathbb{R}$. Also there exists a constant matrix $B$ such that $e^{TB}=X(T)$ and a $T$-periodic  matrix $P(t)$, so that, for all $t$,
        \[ X(t)=P(t)e^{Bt}. \]
  \end{Theorem}

  \begin{Lemma}
     Let be $\lambda$ a Floquet multiplier for (\ref{linearsystem}) and $\lambda=e^{\mu T}$, then there exists a nontrivial solution $x(t)=e^{\mu t}p(t)$, with $p(t)$ a $T$-periodic function. Moreover, $x(t+T)=\lambda x(t)$.
  \end{Lemma}
  Thus, the Floquet multipliers lead to the following characterization:
  \begin{itemize}
     \item If $|\lambda|<1\Leftrightarrow\textrm{Re}(\mu)<0$ then $x(t)\rightarrow 0$ as $t\rightarrow \infty$.
     \item If $|\lambda|=1\Leftrightarrow\textrm{Re}(\mu)=0$ then $x(t)$ is a pseudo-periodic, bounded solution. In particular when $\lambda=1$ then $x(t)$ is $T$-periodic and when $\lambda=-1$ then $x(t)$ is $2T$-periodic.
     \item If $|\lambda|>1$ then $\textrm{Re}(\mu)>0$ and therefore $x(t)\rightarrow \infty$ as $t
         \rightarrow \infty$, an unbounded solution.
  \end{itemize}

  \begin{Lemma}
     If  the Floquet multipliers satisfy $\lambda_{1}\neq\lambda_{2}$,  then the equation (\ref{linearsystem}) has two linearly independent solutions
         \[ x_{1}(t)=p_{1}(t)e^{\mu_{1}t}\,,\qquad x_{2}(t)=p_{2}(t)e^{\mu_{2}t}\,, \]
         where $p_{1}(t)$ and $p_{2}(t)$ are $T$-periodic functions and $\mu_{1}$ and $\mu_{2}$ are the respective Floquet exponents.
  \end{Lemma}
  In this way, the stability of the solution to (\ref{linearsystem}) is
  \begin{itemize}
     \item Asymptotically stable if $|\lambda_{i}|<1$ for $i=1,2$.
     \item Lyapunov stable if $|\lambda_{i}|\leq1$ for $i=1,2$, or if $|\lambda_{i}|=1$ and  the algebraic multiplicity equals the geometric multiplicity.
     \item Unstable if $|\lambda_{i}|>1$ for at least one $i$, or if $|\lambda_{i}|=1$ and the algebraic multiplicity is greater than the geometric multiplicity.
  \end{itemize}

   A particular case for the equation (\ref{linearsystem}) is the so-called Hill's equation, namely the periodic linear second order differential equation:
   \begin{equation}\label{ecapen}
      \ddot{z}+f(t)z=0\,,
   \end{equation}
   where $f(t)$ is a $\pi$-periodic function. Equation (\ref{ecapen}), as a first order system, is
  \begin{equation}\label{linealecapen}
      \dot{\mathbf{v}}=A(t)\mathbf{v}\,,\qquad \mathbf{v}=\left(\begin{array}{c}x \\y\\ \end{array}\right)\,,\qquad
      A(t)=\left(
               \begin{array}{cc}
                 0 & 1 \\
                 -f(t) & 0 \\
               \end{array}
             \right)\,.
  \end{equation}
  Let (\ref{fundmatrix}) be a fundamental matrix solution of (\ref{linealecapen}). The monodromy matrix corresponds to $X(\pi)$ and, as before, let $\lambda_{1}$ and $\lambda_{2}$ be the Floquet multipliers associated to (\ref{linealecapen}) and $\mu_{1}$ and $\mu_{2}$ be the corresponding Floquet exponents.

  In this paper we assume that $f(t)$ is an even function; from Floquet theory we know that $x_{1}$ is an even and $x_{2}$ is an odd function. Also, the trace of $A(t)$ vanishes and
  \[
     \lambda_{1}\lambda_{2}=e^{\int_{0}^{\pi}\textrm{tr}(A(t))dt}=1\,.
  \]
  Assuming $\mu_1=a+ib$ is the Floquet exponent  corresponding to the Floquet multiplier $\lambda_{1}$, the general solution to (\ref{linealecapen}) is characterized as follows:
  \begin{itemize}
     \item Elliptic type: $\lambda_{1}\in\mathbb{C}\setminus\mathbb{R}$, with $|\lambda_{1}|=1$ (and $\lambda_{2}=\bar{\lambda}_{1}$). The general solution is pseudo-periodic and can be written as
         \[
            \mathbf{v}(t)=c_{1}\textrm{Re}\left(\mathbf{p}(t)e^{ibt/\pi}\right)+c_{2}\textrm{Im}\left(\mathbf{p}(t)e^{ibt/\pi}\right)\,.
         \]
         The origin is Lyapunov stable.
     \item Parabolic type: $\lambda_{1}=\lambda_{2}=\pm 1$.
        \begin{itemize}
           \item If $\lambda_{1}=1$ and there are two linearly independent eigenvectors of the monodromy matrix, the general solution is
               \[\mathbf{v}(t)=c_{1}\mathbf{p}_{1}(t)+c_{2}\mathbf{p}_{2}(t)\,,\]
               and is $\pi$-periodic and Lyapunov stable.
           \item If $\lambda_{1}=1$ and there is just one eigenvector associated to this eigenvalue, thus the general solution is
               \[\mathbf{v}(t)=(c_{1}+c_{2}t)\mathbf{p}_{1}(t)+c_{2}\mathbf{p}_{2}(t)\,,\]
               and is unstable.
           \item If $\lambda_{2}=-1$ and there are two linearly independent eigenvectors of the monodromy matrix, the general solution is
               \[ \mathbf{v}(t)=c_{1}\mathbf{p}_{1}(t)e^{it}+c_{2}\mathbf{p}_{2}(t)e^{it}\,,\]
               and is $2\pi$-periodic and Lyapunov stable.
           \item If $\lambda_{1}=-1$ and there is just one eigenvector associated to this eigenvalue, thus the general solution is
               \[\mathbf{v}(t)=(c_{1}+c_{2}t)\mathbf{p}_{1}(t)e^{it}+c_{2}\mathbf{p}_{2}(t)e^{it}\,,\]
               and is unstable.
        \end{itemize}
     \item Hyperbolic type: $\lambda_{1}\in\mathbb{R}$, but $|\lambda_{1}|\neq 1$ (and $\lambda_{2}=1/\lambda_{1}$).
         \begin{itemize}
            \item If $\lambda_{1}>1$, then the solution is
            \[
            \mathbf{v}(t)=c_{1}\mathbf{p}_{1}(t)e^{\mu_{1} t}+c_{2}\mathbf{p}_{2}(t)e^{-\mu_{1} t}\,.
            \]
            \item If $\lambda_{1}<-1$, then the solution is
            \[
            \mathbf{v}(t)=c_{1}\mathbf{p}_{1}(t)e^{\mu_{1} t}e^{it}+c_{2}\mathbf{p}_{2}(t)e^{-\mu_{1} t}e^{it}\,.
            \]
         \end{itemize}
         Thus, the origin is unstable.
  \end{itemize}

\medskip

A useful result from Rafael Ortega \cite{Ortega}  considers the nonlinear Hill equation
\begin{equation}\label{ecortega}
   y''+a(t)y+c(t)y^{2n-1}+d(t,y)=0,
\end{equation}
with $n\geq2$,  where the functions $a,c:\mathbb{R}\rightarrow\mathbb{R}$ are continuous,  $T$-periodic, and
$\int_{0}^{T}|c(t)|\neq 0$,
and the function $d:\mathbb{R}\times(-\epsilon,\epsilon)\rightarrow\mathbb{R}$, for $\epsilon>0$, is continuous, has continuous derivatives of all orders respect to $y$, is $T$-periodic function respect to $t$, and
$d(t,y)=\mathcal{O}\left(|y|^{2n}\right)$  as
$y\rightarrow0$ uniformly with respect to $t\in\mathbb{R}$.

The linear part around the solution $y=0$ of (\ref{ecortega}) is
\begin{equation}\label{ortegalineal}
   y''+a(t)y=0\,.
\end{equation}

\begin{Theorem}[Ortega's Theorem]
   Assume the following:
   \begin{enumerate}
      \item The equation (\ref{ortegalineal}) is stable.
      \item $c\geq0$ or $c\leq0$.
   \end{enumerate}
   Then $y=0$ is a stable solution of (\ref{ecortega}).
\end{Theorem}

\subsection*{Acknowledgements} We thank to the anonymous referees, their remarks and suggestions help us to improve this paper. Research of M.G. was partially supported by NSF grant  DMS-1515851.
M. L. gratefully acknowledges support by the NSF grant DMS-1412542.
The fourth author (EPC) has received partial support by the Asociaci\'on Mexicana de Cultura A.C.
Parts of this work have been done while the authors visited CIMAT, Guanajuato, LFP and EPC visit Yeshiva University and M.G. visited UAM-I in Mexico City. All authors are grateful for the hospitality of these institutions.

\newpage


\end{document}